\documentclass[a4paper,11pt,]{article}
\usepackage{}
\usepackage{amsfonts}
\usepackage{graphics}
\usepackage{amsthm}
\usepackage{hyperref}
\usepackage{mathrsfs}
\usepackage{fancyhdr}
\usepackage{mathrsfs}
\usepackage{amsthm}
\usepackage{mathrsfs}
\usepackage{amsfonts}
\usepackage{amsmath}
\usepackage{amssymb}
\usepackage{amsthm}
\usepackage{enumerate}
\usepackage[mathscr]{eucal}
\usepackage{eqlist}
\usepackage{graphicx}
\usepackage{color}

\usepackage[body={16.7true cm, 24.6true cm}]{geometry}

\newtheorem{Theorem}{Theorem}[section]

\newtheorem{Lemma}[Theorem]{Lemma}
\newtheorem{Corollary}[Theorem]{Corollary}
\newtheorem{Remark}[Theorem]{Remark}

\numberwithin{equation}{section} \allowdisplaybreaks
\usepackage{mathdots}

\renewcommand\abstract{{\bf Abstract}}

\begin{document}
\title{Critical points of solutions to a quasilinear elliptic equation with nonhomogeneous Dirichlet boundary conditions\footnote{\footnotesize The work is supported by National Natural Science Foundation of China (No.11401307, No.11401310), High level talent research fund of Nanjing Forestry University (G2014022) and Postgraduate Research \& Practice Innovation Program of Jiangsu Province (KYCX17\_0321). The second author is sponsored by Qing Lan Project of Jiangsu Province.}}

\author{Haiyun Deng$^{1}$\footnote{\footnotesize Corresponding author E-mail: haiyundengmath1989@163.com, Tel.: +86 15877935256}, Hairong Liu$^{2}$, Long Tian$^{1}$  \\[12pt]
\small \emph {$^{1}$School of Science, Nanjing University of Science and Technology, Nanjing, Jiangsu, 210094, China;}\\
\small \emph {$^{2}$School of Science, Nanjing Forestry University, Nanjing, Jiangsu, 210037, China;}}
\date{}
\maketitle

\renewcommand{\labelenumi}{[\arabic{enumi}]}

\begin{abstract}{\bf:} \footnotesize
 In this paper, we mainly investigate the critical points associated to solutions $u$ of a quasilinear elliptic equation with nonhomogeneous Dirichlet boundary conditions in a connected domain $\Omega$ in $\mathbb{R}^2$. Based on the fine analysis about the distribution of connected components of a super-level set $\{x\in \Omega: u(x)>t\}$ for any $\mathop {\min}_{\partial\Omega}u(x)<t<\mathop {\max}_{\partial\Omega}u(x)$, we obtain the geometric structure of interior critical points of $u$. Precisely, when $\Omega$ is simply connected, we develop a new method to prove $\Sigma_{i = 1}^k {{m_i}}+1=N$, where $m_1,\cdots,m_k$ are the respective multiplicities of interior critical points $x_1,\cdots,x_k$ of $u$ and $N$ is the number of global maximal points of $u$ on $\partial\Omega$. When $\Omega$ is an annular domain with the interior boundary $\gamma_I$ and the external boundary $\gamma_E$, where $u|_{\gamma_I}=H,~u|_{\gamma_E}=\psi(x)$ and $\psi(x)$ has $N$ local (global) maximal points on $\gamma_E$. For the case $\psi(x)\geq H$ or $\psi(x)\leq H$ or $\mathop {\min}\limits_{\gamma_E}\psi(x)<H<\mathop {\max}\limits_{\gamma_E}\psi(x)$, we show that $\Sigma_{i = 1}^k {{m_i}} \le N$ (either $\Sigma_{i = 1}^k {{m_i}}=N$ or $\Sigma_{i = 1}^k {{m_i}}+1=N$).

\end{abstract}



\section{Introduction and main results}
~~~~~In this paper we mainly investigate the interior critical points of solutions to the following a quasilinear elliptic equation
\begin{equation}\label{1.1}\begin{array}{l}
\begin{array}{l}
Lu=\sum\limits_{i,j=1}^{2}a_{ij}(\nabla u)\frac{\partial^2 u}{\partial x_i \partial x_j}=0~~~\mbox{in}~\Omega,
\end{array}
\end{array}\end{equation}
where $\Omega$ is a bounded, smooth and connected domain in $\mathbb{R}^{2}$,  $a_{ij}$ is smooth and $L$ is uniformly elliptic in $\Omega$.

The subject of critical points is a significant research topic for solutions of elliptic equations. Until now, there are many results about the critical points.  In 1992 Alessandrini and Magnanini \cite{AlessandriniMagnanini1} studied the geometric structure of the critical set of solutions to a semilinear elliptic equation in a planar nonconvex domain, whose boundary is composed of finite simple closed curves. They deduced that the critical set is made up of finitely many isolated critical points. In 1994, Sakaguchi \cite{Sakaguchi2} considered the critical points of solutions to an obstacle problem in a planar, bounded, smooth and simply connected domain. He showed that if the number of critical points of the obstacle is finite and the obstacle has only N local (global) maximum points, then the inequality $\Sigma_{i = 1}^k {{m_i}}+1 \le N$ (the equality $\Sigma_{i = 1}^k {{m_i}}+1=N$) holds for the critical points of one solution in the noncoincidence set, where $m_1,m_2,\cdots,m_k$ are the multiplicities of critical points $x_1,x_2,\cdots,x_k$ respectively. In 2012 Arango and G\'{o}mez \cite{ArangoGomez2} considered critical points of the solutions to a quasilinear elliptic equation with Dirichlet boundary condition in strictly convex and nonconvex planar domains respectively. If the domain is strictly convex and $u$ is a negative solution, they proved that such a critical point set has exactly one nondegenarate critical point. Moreover, they obtained the similar results of a semilinear elliptic equation in a planar annular domain, whose boundary has nonzero curvature. See \cite{AlessandriniMagnanini2,Cecchini,ChenHuang,Enciso1,Enciso2,Enciso3,Grieser,Kawohl,Kraus,Magnanini,PucciSerrin} for related results.

 Concerning the Neumann and Robin boundary value problems, there exist a few results about the critical points of solutions to elliptic equations. In 1990, Sakaguchi \cite{Sakaguchi1} proved that a solution of Poisson equation with Neumann or Robin boundary condition has exactly one critical point in a planar domain. In 2017, Deng, Liu and Tian \cite{Deng2} showed the nondegeneracy and uniqueness of the critical point of a solution to prescribed constant mean curvature equation with Neumann or Robin boundary condition in a smooth, bounded and strictly convex domain $\Omega$ of $\mathbb{R}^n(n\geq2)$.

 For the higher dimensional cases. Under the assumption of the existence of a semi-stable solution of Poisson equation $-\triangle u=f(u)$, Cabr\'{e} and Chanillo \cite{CabreChanillo} showed that the solution $u$ has exactly one nondegenerate critical point in bounded, smooth and convex domains of $\mathbb{R}^{n}(n\geq2)$. Deng, Liu and Tian \cite{Deng1} investigated the geometric structure of critical points of solutions to mean curvature equations with Dirichlet boundary condition and showed that the critical point set $K$ has exactly one nondegenerate critical point in a strictly convex domain of $\mathbb{R}^{n}(n\geq2)$ and $K$ has (respectively, has no) a rotationally symmetric critical closed surface $S$ in a concentric (respectively, an eccentric) spherical annulus domain of $\mathbb{R}^{n}(n\geq3)$.

However, as we know, there is few work on the critical points of solutions to quasilinear elliptic equations with nonhomogeneous Dirichlet boundary conditions. The goal of this paper is to study the critical points of solutions to a quasilinear elliptic equation with nonhomogeneous Dirichlet boundary conditions. Our main results are as follows.

\begin{Theorem}\label{th1.1} 
 Let $\Omega$ be a bounded, smooth and simply connected domain in $\mathbb{R}^{2}$. Suppose that $\psi(x)\in C^1(\overline{\Omega})$ and that $\psi$ has $N$ local maximal points on $\partial\Omega.$  Let $u$ be a non-constant solution of the following boundary value problem
\begin{equation}\label{1.2}\begin{array}{l}
\left\{
\begin{array}{l}
\sum\limits_{i,j=1}^{2}a_{ij}(\nabla u)\frac{\partial^2 u}{\partial x_i \partial x_j}=0~~~\mbox{in}~\Omega,\\
u=\psi(x)~~~\mbox{on}~ \partial\Omega.
\end{array}
\right.
\end{array}\end{equation}
Then $u$ has finite interior critical points, denoting by $x_1,x_2,\cdots,x_k$, and the following inequality holds
\begin{equation}\label{1.3}\begin{array}{l}
\begin{array}{l}
\sum\limits_{i = 1}^k {{m_i}}+1 \le N,
\end{array}
\end{array}\end{equation}
where $m_1,m_2,\cdots,m_k$ are the multiplicities of critical points $x_1,x_2,\cdots,x_k$ respectively.
\end{Theorem}

\begin{Theorem}\label{th1.2} 
Let $\Omega$ be a bounded, smooth and simply connected domain in $\mathbb{R}^{2}$. Suppose that $\psi(x)\in C^1(\overline{\Omega})$ and that $\psi$ has only $N$ global maximal points and $N$ global minimal points on $\partial\Omega$, i.e., all the maximal and minimal points of $\psi$ are global. Let $u$ be a non-constant solution of (\ref{1.2}). Then $u$ has finite interior critical points and
\begin{equation}\label{1.4}\begin{array}{l}
\begin{array}{l}
\sum\limits_{i = 1}^k {{m_i}}+1= N,
\end{array}
\end{array}\end{equation}
where $m_i$ is as in Theorem \ref{th1.1}.
\end{Theorem}

\begin{Theorem}\label{th1.3} 
 Let $\Omega$ be a bounded smooth annular domain with the interior boundary $\gamma_I$ and the external boundary $\gamma_E$ in $\mathbb{R}^{2}$. Suppose that $\psi(x)\in C^1(\overline{\Omega}),$ $H$ is a given constant, $\psi(x)\geq H$ and that $\psi$ has $N$ local maximal points on $\gamma_E.$  Let $u$ be a non-constant solution of the following boundary value problem
\begin{equation}\label{1.5}\begin{array}{l}
\left\{
\begin{array}{l}
\sum\limits_{i,j=1}^{2}a_{ij}(\nabla u)\frac{\partial^2 u}{\partial x_i \partial x_j}=0~~~\mbox{in}~\Omega,\\
u|_{\gamma_I}=H,~~u|_{\gamma_E}=\psi(x).
\end{array}
\right.
\end{array}\end{equation}
Then $u$ has finite interior critical points and
\begin{equation}\label{1.6}\begin{array}{l}
\begin{array}{l}
\sum\limits_{i = 1}^k {{m_i}} \le N,
\end{array}
\end{array}\end{equation}
where $m_i$ is as in Theorem \ref{th1.1}.
\end{Theorem}

\begin{Theorem}\label{th1.4} 
Let $\Omega$ be a bounded smooth annular domain with the interior boundary $\gamma_I$ and the external boundary $\gamma_E$ in $\mathbb{R}^{2}$. Suppose that $\psi(x)\in C^1(\overline{\Omega}),~\psi(x)\geq H$ and that $\psi$ has only $N$ global maximal points and $N$ global minimal points on $\gamma_E$, i.e., all the maximal and minimal points of $\psi$ are global. Let $u$ be a non-constant solution of (\ref{1.5}). Then $u$ has finite interior critical points, and either
\begin{equation}\label{1.7}\begin{array}{l}
\begin{array}{l}
\sum\limits_{i = 1}^k {{m_i}}= N,
\end{array}
\end{array}\end{equation}
or
\begin{equation}\label{1.8}\begin{array}{l}
\begin{array}{l}
\sum\limits_{i = 1}^k {{m_i}}+1= N,
\end{array}
\end{array}\end{equation}
where $m_i$ is as in Theorem \ref{th1.1}.
\end{Theorem}

\begin{Theorem}\label{th1.5} 
 Let $\Omega$ be a bounded smooth annular domain with the interior boundary $\gamma_I$ and the external boundary $\gamma_E$ in $\mathbb{R}^{2}$. Suppose that $\psi(x)\in C^1(\overline{\Omega}),$ $H$ is a given constant, $\mathop {\min}\limits_{\gamma_E}\psi(x)<H<\mathop {\max}\limits_{\gamma_E}\psi(x)$ and that $\psi$ has $N$ local maximal points on $\gamma_E.$ Let $u$ be a non-constant solution of (\ref{1.5}). Then $u$ has finite interior critical points and
\begin{equation}\label{1.9}\begin{array}{l}
\begin{array}{l}
\sum\limits_{i = 1}^k {{m_i}} \le N,
\end{array}
\end{array}\end{equation}
where $m_i$ is as in Theorem \ref{th1.1}.
\end{Theorem}

\begin{Theorem}\label{th1.6} 
Let $\Omega$ be a bounded smooth annular domain with the interior boundary $\gamma_I$ and the external boundary $\gamma_E$ in $\mathbb{R}^{2}$. Suppose that $\psi(x)\in C^1(\overline{\Omega}),$ $\mathop {\min}\limits_{\gamma_E}\psi(x)<H<\mathop {\max}\limits_{\gamma_E}\psi(x)$ and that $\psi$ has only $N$ global maximal points and $N$ global minimal points on $\gamma_E,$ i.e., all the maximal and minimal points of $\psi$ are global. Let $u$ be a non-constant solution of (\ref{1.5}). Then $u$ has finite interior critical points, and either
\begin{equation}\label{1.10}\begin{array}{l}
\begin{array}{l}
\sum\limits_{i = 1}^k {{m_i}}= N,
\end{array}
\end{array}\end{equation}
or
\begin{equation}\label{1.11}\begin{array}{l}
\begin{array}{l}
\sum\limits_{i = 1}^k {{m_i}}+1= N,
\end{array}
\end{array}\end{equation}
where $m_i$ is as in Theorem \ref{th1.1}.
\end{Theorem}

\begin{Remark}\label{re1.7}
 In particular, when ${a_{ij}(\nabla u)=\frac{1}{\sqrt{1+|\nabla u|^2}}(\delta_{ij}-\frac{u_{x_i}u_{x_j}}{1+|\nabla u|^2})}$, then the minimal surface equation $\mbox {div}(\frac{\nabla u}{\sqrt{1+|\nabla u|^2}})=0$ in a bounded smooth domain is a particular example of a quasilinear elliptic equation in (\ref{1.1}).
\end{Remark}

 For the sake of clarity, we now explain the key ideas which are used to prove the main results. We prove (\ref{1.3}) and (\ref{1.6}) by induction and the strong maximum principle. On the other hand, we develop a new method to prove (\ref{1.4}) and (\ref{1.7}), which is different from the method in \cite{Sakaguchi2}. In \cite{Sakaguchi2}, the author divided the proof into two cases: all interior critical values are equal, i.e., $u(x_1)=\cdots=u(x_k)$, and all interior critical values are not totally equal. However, when $\psi(x)$ has only $N$ global maximal points and $N$ global minimal points on the boundary $\partial\Omega,$ we prove that all interior critical values are equal. We obtain (\ref{1.4}) and (\ref{1.7}) by showing that there are the following three ``{\bf just right}''s:

 (i) the first ``{\bf just right}'' is that the critical values for all interior critical points are equal (i.e., $u(x_1)=u(x_2)=\cdots=u(x_k)=t$ for some $t$);

 (ii) the second ``{\bf just right}'' is that all the critical points $x_1,x_2,\cdots,x_k$ together with the corresponding level lines of $\{x\in \Omega : u(x)=t\}$ clustering round these points form a connected set;

 (iii) the third ``{\bf just right}'' is that every simply connected component $\omega$ of $\{x\in \Omega: u(x)>t\}$ ($\{x\in \Omega: u(x)<t\}$) has exactly one global maximal (minimal) point on the boundary $\partial\Omega$.

The rest of this paper is organized as follows. In Section 2, we investigate the geometric structure of interior critical points of solutions in a bounded, smooth and simply connected domain in $\mathbb{R}^2$. We show that if $\psi(x)$ has only N local (global) maximal points on $\partial\Omega$, then $\Sigma_{i = 1}^k {{m_i}}+1 \le N$ ($\Sigma_{i = 1}^k {{m_i}}+1=N$) holds for the interior critical points of a solution $u$. We develop a new method to prove $\Sigma_{i = 1}^k {{m_i}}+1=N$, we show the three ``just right''s. In Section 3, we study the geometric structure of interior critical points of solutions in a bounded smooth annular domain with the interior boundary $\gamma_I$ and the external boundary $\gamma_E$ in $\mathbb{R}^2$, where $u|_{\gamma_I}=H,~u|_{\gamma_E}=\psi(x),~\psi(x)\geq H$ and $\psi$ has $N$ local (global) maximal points on $\gamma_E.$ We deduce $\Sigma_{i = 1}^k {{m_i}} \le N$ ($\Sigma_{i = 1}^k {{m_i}}=N$ or $\Sigma_{i = 1}^k {{m_i}}+1=N$). In Section 4, we investigate the case of $\mathop {\min}\limits_{\gamma_E}\psi(x)<H<\mathop {\max}\limits_{\gamma_E}\psi(x)$, where $\psi$ has $N$ local (global) maximal points on $\gamma_E$ and show the same results as in Section 3.

\section{The case of simply connected domains}
\subsection{Proof of Theorem \ref{th1.1}}
~~~~~~~In order to prove Theorem \ref{th1.1}, we need the following basic lemmas.
\begin{Lemma}\label{le2.1} 
Let $u$ be a non-constant solution of (\ref{1.2}). For any  $t\in (\mathop{\min }\limits_{\overline \Omega}u,\mathop{\max }\limits_{\overline \Omega}u),$ we have that any connected component of $\{x\in \Omega : u(x)>t\}$ and $\{x\in \Omega : u(x)<t\}$ is simply connected, which has to meet the boundary $\partial \Omega.$
\end{Lemma}
\begin{proof}[Proof] Let $A$ be a connected component of $\{x\in \Omega : u(x)>t\}$ and $\alpha$ be a non-equivalent simple closed curve in $A.$ By the Jordan curve theorem there exists a bounded domain $B$ with $\partial B=\alpha.$ Since $\Omega$ is simply connected, then $B$ is contained in $\Omega.$ The strong maximum principle implies that $u>t$ in domain $B$. It shows that $B$ is contained in $A$, namely $A$ is simply connected. The strong maximum principle shows that $u$ obtain its maximum points and minimal points on boundary $\partial\Omega,$ therefore the connected component $A$ has to meet the boundary $\partial \Omega.$ The proof of the case of $\{x\in \Omega : u(x)<t\}$ is similar.
\end{proof}

\begin{Lemma}\label{le2.2}
 Suppose that $x_0$ is an interior critical point of $u$ in $\Omega$ and that $m$ is the multiplicity of $x_0.$ Then $m+1$ distinct connected components of $\{x\in \Omega :u(x)>u(x_0)\}$ and $\{x\in \Omega :u(x)<u(x_0)\}$ cluster around the point $x_0$ respectively.
\begin{proof}[Proof] According to the results of Hartman and Wintner \cite{Hartman}, in a neighborhood of $x_0$ the level line $\{x\in \Omega : u(x)=u(x_0)\}$ consists of $m+1$ simple arcs intersecting at $x_0$. By the results of Lemma \ref{le2.1}, there exist $m+1$ distinct connected components of $\{x\in \Omega :u(x)>u(x_0)\}$ and $\{x\in \Omega :u(x)<u(x_0)\}$ clustering around the point $x_0$ respectively. This completes the proof.
\end{proof}
\end{Lemma}

\begin{Lemma}\label{le2.3}
Suppose that $u$ is a non-constant solution to (\ref{1.2}). Then $u$ has finite interior critical points in $\Omega$.
\begin{proof}[Proof]We set up the usual contradiction argument. Suppose that $u$ has infinite interior critical points in $\Omega,$ denoting by $x_1,x_2,\cdots$. The results of Lemma \ref{le2.1} and Lemma \ref{le2.2} show that there exists infinite connected components of $\{x\in \Omega :u(x)>u(x_i)\}$ and $\{x\in \Omega :u(x)<u(x_i)\} (i=1,2,\cdots)$. The strong maximum principle implies that there exists at least a maximum point and minimal point on $\partial\Omega$ for any connected components of $\{x\in \Omega :u(x)>u(x_i)\}$ and $\{x\in \Omega :u(x)<u(x_i)\} (i=1,2,\cdots)$ respectively. Therefore there exists infinite maximal points and minimal points on $\partial\Omega,$ this contradicts with the assumption. This completes the proof.
\end{proof}
\end{Lemma}

\begin{Lemma}\label{le2.4}
Let $x_1,x_2,\cdots,x_k$ be the interior critical points of $u$ in $\Omega$. Suppose that $u(x_1)=u(x_2)=\cdots=u(x_k)=t$ for some $t\in \mathbb{R},$ where $m_1,m_2,\cdots,m_k$ are the multiplicities of critical points $x_1,x_2,\cdots,x_k$ respectively. We set $M_1$ and $M_2$ as the number of the connected components of the super-level set $\{x\in\Omega : u(x)>t\}$ and the sub-level set $\{x\in\Omega : u(x)<t\}$ respectively. Suppose that all the critical points $x_1,x_2,\cdots,x_k$ together with the corresponding level lines of $\{x\in \Omega : u(x)=t\}$ clustering round these points form $q$ connected sets, where $q\geq 1$. Then
 \begin{equation}\label{2.1}\begin{array}{l}
M_1\geq \sum\limits_{i = 1}^k {{m_i}}+1,~~~M_2\geq \sum\limits_{i = 1}^k {{m_i}}+1,
\end{array}\end{equation}
and
\begin{equation}\label{2.2}\begin{array}{l}
M_1+M_2=2\sum\limits_{i = 1}^k {{m_i}}+q+1.
\end{array}\end{equation}

\begin{proof}[Proof] We divide the proof into two cases.

 (i) Case 1: When $q=1,$ by induction. Since the number of the connected components of the super-level set $\{x\in\Omega : u(x)>t\}$ equals the number of the connected components of the sub-level set $\{x\in\Omega : u(x)<t\}.$ Without loss of generality, we only estimate the number of the connected components of the super-level set $\{x\in\Omega : u(x)>t\}.$  By induction. When $k=1,$ the result holds by Lemma \ref{le2.2}. Assume that $1\leq k\leq n$ the connected set, which consists of $k$ critical points and the connected components clustering round these points, contains exactly $\sum_{i = 1}^k {{m_i}}+1$ components of the super-level set $\{x\in\Omega : u(x)>t\}$. Let $k=n+1.$ Let $A$ be the set which consists of the points $x_1,x_2,\cdots,x_{n+1}$ together with the respective components clustering round these points. We may assume that the points $x_1,x_2,\cdots,x_n$ together with the respective components clustering round these points form a connected set, denotes by $B.$ By Lemma \ref{le2.1} we know that $A$ cannot surround a component of $\{x\in\Omega : u(x)<t\}.$  Up to renumbering, therefore there is only one component of $\{x\in\Omega : u(x)>t\}$ whose boundary $\gamma$ contains both $x_n$ and $x_{n+1}$. Next we give the distribution for the level lines of $\{x\in\Omega : u(x)=t\}.$
\begin{center}
  \includegraphics[width=6cm,height=3.6cm]{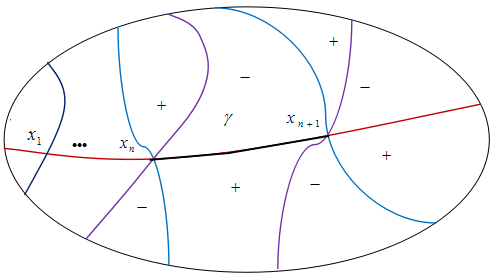}\\
  \scriptsize {\bf Figure 1.}~~The distribution for the level lines of $\{x\in\Omega : u(x)=t\}$.
\end{center}
 Since both $A$ and $B$ are connected. By using Lemma \ref{le2.2} and the inductive assumption to $B,$ then we know that $A$ contains exactly
 $$ (\sum\limits_{i = 1}^n {{m_i}}+1)+(m_{n+1}+1)-1=\sum\limits_{i = 1}^{n+1} {{m_i}}+1$$
connected components of the super-level set $\{x\in\Omega : u(x)>t\}$. This completes the proof of case 1.

(ii) Case 2: When $q\geq 2.$ Since the number of connected sets of the level lines $\{x\in \Omega : u(x)=t\}$ together with $x_1,x_2,\cdots,x_k$ increases one leading the number of connected components of $\{x\in\Omega : u(x)>t\}$ or $\{x\in\Omega : u(x)<t\}$ also increases one, i.e., if the number of connected components of $\{x\in\Omega : u(x)>t\}$ increased by 1, then the number of connected components of $\{x\in\Omega : u(x)<t\}$ unchanged, and vice versa. Figure 2 pictures the changing of connected components of $\{x\in\Omega : u(x)>t\}$ or $\{x\in\Omega : u(x)<t\}$.
\begin{center}
  \includegraphics[width=12cm,height=3.9cm]{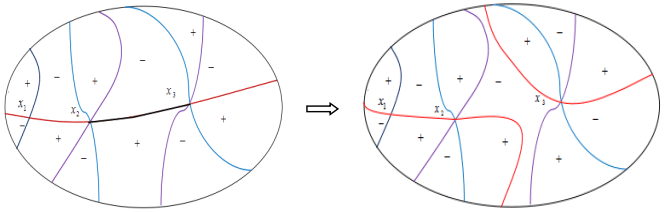}\\
  \scriptsize {\bf Figure 2.}~~The distribution for the connected components of $\{x\in\Omega : u(x)>t\}$ and $\{x\in\Omega : u(x)<t\}$.
\end{center}
Now we put
 \begin{equation*}\begin{array}{l}
M_1:=\sharp\Big\{\mbox{the\ connected\ components\ of\ the\ super-level\ set} ~ \{x\in\Omega : u(x)>t\} \Big\},\\
M_2:=\sharp\Big\{\mbox{the\ connected\ components\ of\ the\ sub-level\ set} ~ \{x\in\Omega : u(x)<t\} \Big\}.
\end{array}\end{equation*}
If all the critical points $x_1,x_2,\cdots,x_k$ together with the level lines of $\{x\in \Omega : u(x)=t\}$ clustering round these points form $q$ connected sets. By the results of case 1, then we have
\begin{equation*}\begin{array}{l}
M_1\geq \sum\limits_{i = 1}^k {{m_i}}+1,~~~M_2\geq \sum\limits_{i = 1}^k {{m_i}}+1,
\end{array}\end{equation*}
and
\begin{equation*}\begin{array}{l}
M_1+M_2=2(\sum\limits_{i = 1}^k {{m_i}}+1)+(q-1)=2\sum\limits_{i = 1}^k {{m_i}}+q+1.
\end{array}\end{equation*}
This completes the proof of case 2.
\end{proof}
\end{Lemma}

We are now ready to present the proof of Theorem \ref{th1.1}.
\begin{proof}[Proof of Theorem \ref{th1.1}]  (i) Case 1: If $u(x_1)=u(x_2)=\cdots=u(x_k)=t$ for some $t\in \mathbb{R}.$ By the results of Lemma \ref{le2.4}, we know that
\begin{equation*}\begin{array}{l}
 \sharp\Big\{\mbox{the\ connected\ components\ of\ the\ super-level\ set} ~ \{x\in\Omega : u(x)>t\} \Big\}\geq \sum\limits_{i = 1}^k {{m_i}}+1.
\end{array}\end{equation*}
Therefore, in this case the super-level set always has at least $\sum\limits_{i = 1}^k {{m_i}}+1$ connected components and at most $\sum\limits_{i = 1}^k {{m_i}}+q$ connected components. Using the strong maximum principle and Lemma \ref{le2.1}, we have that $u$ exists at least $\sum\limits_{i = 1}^k {{m_i}}+1$ local maximal points on $\partial\Omega.$ Hence, we have
\[ \sum\limits_{i = 1}^k {{m_i}}+1\leq N.\]

(ii) Case 2: The values at critical points $x_1,\cdots,x_k$ are not totally equal. Without loss of generality, we may suppose that
\begin{equation}\label{2.3}
\begin{split}
u(x_1)&=\cdots=u(x_{j_1})<u(x_{j_1+1})=\cdots=u(x_{j_2})<\cdots<\cdots\\
&<u(x_{j_{n-1}+1})=\cdots=u(x_{j_{n}}),
\end{split}
\end{equation}
where $x_1,\cdots,x_{j_1},\cdots, x_{j_2},\cdots, x_{j_{n}}$ are different critical points in $\Omega$, $j_{n}=k$ and $n\geq 2.$ Now we put
\begin{equation*}\begin{array}{l}
E_{j}:=\Big\{\omega : \mbox{open\ set}~ \omega ~ \mbox{is\ a\ connected\ component\ of} ~ \{x\in \Omega: u(x)>u(x_j)\} \Big\}~(j=j_1,j_2,\cdots,j_n),
\end{array}\end{equation*}
and
\begin{equation*}\begin{array}{l}
F_{j_n}:=\Big\{\omega : \mbox{open\ set}~ \omega~ \mbox{is\ a\ connected\ component\ of} ~ \{x\in \Omega: u(x)<u(x_{j_1})\}~\mbox{or} ~\omega~ {is\ a} \\ ~~~~~~~~~~~~ \mbox{connected\ component\ of} ~ \{x\in \Omega: u(x_{j_i})<u(x)<u(x_{j_{i+1}})~\mbox{for\ some}~1\leq i\leq n-1 \}\Big\}.
\end{array}\end{equation*}
According to the definition, we know that $F_{j_n}$ consists of disjoint components. Denotes
\[|F_{j_n}|:=\sharp\{\omega : \omega~ \mbox{is\ a\ connected\ component\ of} ~F_{j_n}\}.\]

To illustrate $F_{j_n}$, let us consider an illustration for $F_{j_2}$. Assume that $u$ has only three critical points $x_{j_1},x_{j_1+1},x_{j_2}$ with respective multiplicity $m_{j_1}=1,m_{j_1+1}=1,m_{j_2}=1$ in $\Omega$ and $u(x_{j_1})<u(x_{j_1+1})=u(x_{j_2})$. The distribution of elements of $F_{j_2}$ as follows:
\begin{center}
  \includegraphics[width=6.6cm,height=3.6cm]{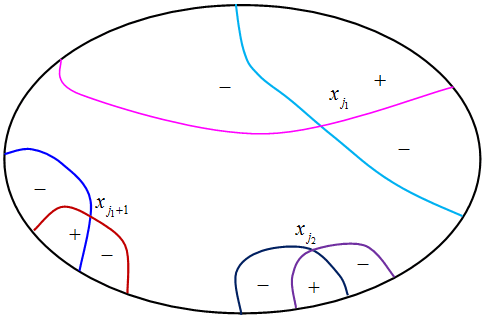}\\
  \scriptsize {\bf Figure 3.}~~The distribution of elements of $F_{j_2}$.
\end{center}
By the definition of $F_{j_n},$ we know $|F_{j_2}|=6.$

Now let us show that $\big|F_{j_s}\big|\geq \sum_{i = 1}^{{j_s}} m_i+1$ by induction on the number $s.$ When $s=1,$ the result holds by case 1. Assume that $\big|F_{j_s}\big|\geq \sum_{i = 1}^{{j_s}} m_i+1$ for $1\leq s\leq n-1.$  Let $s=n.$ Then, by (\ref{2.3}) and the definition of $E_{j}$, we have
\[\{x_{j_{n-1}+1},\cdots,x_{j_{n}}\}\subset \bigcup \limits_{\omega  \in {E_{{j_{n-1}}}}} \omega,\]
where $\omega$ is a connected component of $\{x\in \Omega: u(x)>u(x_{j_{n-1}})\}.$

Let us assume that $\{x_{j_{n-1}+1},\cdots,x_{j_{n}}\}$ are contained in exactly $\widetilde{q}$ components $\omega_1,\cdots,\omega_{\widetilde{q}}.$ Then $x_{j_{n-1}+1},\cdots,x_{j_{n}}$ together with the corresponding level lines of $\{x\in \Omega: u(x)=u(x_{j_{n}})\}$ clustering round these points at least form $\widetilde{q}$ connected sets. By Lemma \ref{le2.4}, we have
\begin{equation*}\begin{array}{l}
M:=\sharp\Big\{\mbox{the\ connected\ components\ of}~ \{x\in \Omega: u(x_{j_{n-1}})<u(x)<u(x_{j_{n}})\}~\mbox{in\ all}~\omega_j~(j=1,\cdots,\widetilde{q})\Big\}\\ ~~~~\geq \sum\limits_{i=j_{n-1}+ 1}^{j_{n}} {{m_i}}+\widetilde{q}.
\end{array}\end{equation*}
By using the definition of $\big|F_{j_{n}}\big|$ and the inductive assumption to $1\leq s\leq n-1,$ then we have
\begin{equation*}\begin{array}{l}
\big|F_{j_{n}}\big|=\big|F_{j_{n-1}}\big|+M\geq \big|F_{j_{n-1}}\big|+(\sum\limits_{i=j_{n-1}+ 1}^{j_{n}} {{m_i}}+\widetilde{q})-\widetilde{q}\geq \sum\limits_{i=1}^{j_{n}} {{m_i}}+1.
\end{array}\end{equation*}
By the strong maximum principle and Lemma \ref{le2.1}, we have that $u$ has at least $\sum\limits_{i = 1}^k {{m_i}}+1$ local minimal points on $\partial\Omega.$ Therefore, we obtain
\[ \sum\limits_{i = 1}^k {{m_i}}+1\leq N.\]
This completes the proof of case 2.
\end{proof}

\subsection{Proof of Theorem \ref{th1.2}}
~~~~~In this subsection, we investigate the geometric structure of interior critical points of a solution in a planar, bounded, smooth and simply connected domain $\Omega$ for the case of $\psi$ having only $N$ global maximal points and $N$ global minimal points on $\partial\Omega.$ We develop a new method to prove $\Sigma_{i = 1}^k {{m_i}}+1=N$, where $N\geq 2$, and we show the three ``just right''s.

\begin{proof}[Proof of Theorem \ref{th1.2}] We divide the proof into five steps.

Step 1, we show that $u$ has at least one interior critical point in $\Omega.$ Suppose by contradiction that $|\nabla u|>0$ in $\Omega$ and $z = \mathop {\min}
\limits_{\partial \Omega}u, Z = \mathop {\max}\limits_{\partial \Omega}u$. The strong maximum principle implies that $u$ has no interior maximum point and minimal point in $\Omega,$ then we have
 $$z<u(x)<Z ~\mbox{for\ any}~ x\in \Omega.$$
According to the assumption of Theorem \ref{th1.2}, let $q_1,\cdots,q_N$ and $p_1,\cdots,p_N$ be the global maximal points and minimal points on $\partial\Omega,$ respectively.

Without loss of generality, we may assume that there only exists two different global maximal points $q_1,q_2$ and global minimal points $p_1,p_2$ on boundary $\partial \Omega.$ Note that $u$ is monotonically decreasing on the connected components of boundary $\partial\Omega$ from one maximal point to the near minimal point. Therefore, by the continuity of level lines $\{x\in \Omega: u(x)=t_0,z<t_0<Z\},$ we know that $\{x\in \Omega: u(x)=Z-\epsilon\}$ ($\{x\in \Omega: u(x)=z+\epsilon\}$) exactly exists two level lines in $\Omega$ for any $\epsilon$ such that $0<\epsilon<Z-z$. This is impossible, because this would imply that either: $u(x)=z~(u(x)=Z)~\mbox{in\ interior\ points\ of}~\Omega$, or: there exists two level lines intersect in $\Omega$, i.e., there exists critical points in $\Omega,$ this contradicts with the assumption $|\nabla u|>0$ in $\Omega.$ This completes the proof of step 1. The figure as shown in Figure 4.
\begin{center}
  \includegraphics[width=7.0cm,height=3.5cm]{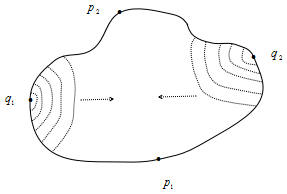}\\
  \scriptsize {\bf Figure 4.}~~The distribution of some level lines $\{x\in \Omega:u(x)=t_0, z<t_0<Z\}.$
\end{center}

Step 2, the first ``just right'': According to Lemma \ref{le2.3}, we assume that the interior critical points of $u$ are $x_1,x_2,\cdots,x_k$. We show that $u(x_1)=u(x_2)=\cdots=u(x_k)=t~\mbox{for\ some} ~ t\in \mathbb{R}$. We set up the usual contradiction argument. We assume that the values at critical points $x_1,\cdots,x_k$ are not totally equal. Without loss of generality, we suppose that $u(x_1)<u(x_2)$ and that $m_1,m_2$ are the respective multiplicities of $x_1,x_2.$ Then, by Lemma \ref{le2.1}, we know that any connected component $B$ of $\{x\in \Omega:u(x)<u(x_1)\}$ has to meet the boundary $\partial\Omega$ and $u(p_1)<u(x_1)$, where $p_1$ is the minimal point of connected component $B$ on $\partial\Omega$. At the same time, we know that any connected component $C$ of $\{x\in \Omega:u(x)<u(x_2)\}$ has to meet the boundary $\partial\Omega$ and $u(x_1)<u(p_2)<u(x_2)$, where $p_2$ is the minimal point of connected component $C$ on $\partial\Omega$ (see Figure 5). Then $u(p_1)\neq u(p_2)$, which contradicts with the assumption of Theorem \ref{th1.2}. This completes the proof of step 2.
\begin{center}
  \includegraphics[width=6cm,height=3.5cm]{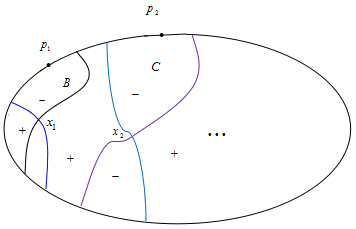}\\
  \scriptsize {\bf Figure 5.} ~~The distribution of the connected components.
\end{center}

Step 3, the second ``just right'': we show that $x_1,x_2,\cdots,x_k$ together with the corresponding level lines of $\{x\in \Omega : u(x)=t\}$ clustering round these points exactly form one connected set. Without loss of generality, we suppose by contradiction that $x_1,x_2,\cdots,x_k$ together with the level lines of $\{x\in \Omega : u(x)=t\}$ clustering round these points form two connected sets.
Therefore, there exists a connected components of $\{x\in \Omega: u(x)<t\}$ (or $\{x\in \Omega: u(x)>t\}$), which meets two parts $\gamma_1,\gamma_2$ of $\partial\Omega$, denoting by $A$ (see Figure 6).
\begin{center}
  \includegraphics[width=6cm,height=3.7cm]{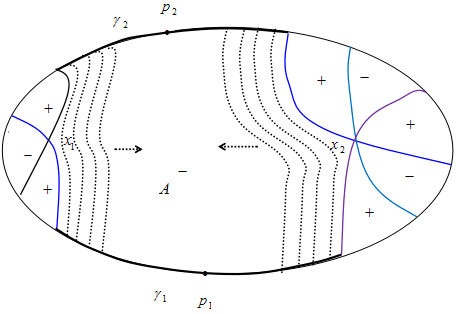}\\
  \scriptsize {\bf Figure 6.}~~The distribution of some level lines $\{x\in \Omega:u(x)=\widetilde{t},z<\widetilde{t}<t \}.$
\end{center}
Note that $u$ is monotonically decreasing on the connected components of boundary $\partial\Omega$ from one maximal point to the near minimal point. Therefore, $\{x\in A: u(x)=t-\epsilon\}$ exactly exists two level lines in $A$ for any $\epsilon$ such that $0<\epsilon<t-z$. This is impossible, because this would imply that either: $u(x)=z ~\mbox{in\ interior\ points\ of }~A$, or: there exists two level lines intersect in $A$, i.e., there exists critical points in $A.$ This completes the proof of step 3.

Step 4, the third ``just right'': we show that every connected component of $\{x\in \Omega: u(x)>t\}$ ($\{x\in \Omega: u(x)<t\}$) has exactly one global maximal (minimal) point on boundary $\partial\Omega.$ In fact, we assume that some connected component $B$ of $\{x\in \Omega: u(x)>t\}$ ($\{x\in \Omega: u(x)<t\}$) exists two global maximal (minimal) points on boundary $\partial\Omega.$ According to $\psi$ has only $N$ global maximal points and $N$ global minimal points on $\partial\Omega$, then there must exist a minimal point $p$ between the two maximal points on $\partial\Omega$ such that $u(p)=z$. Since $u(x)>t>z$ in $B$, by the continuity of solution $u$, this contradicts with the definition of connected component $B$. This completes the proof of step 4.

Step 5, By the results of step 3 and the results of case 1 in Lemma \ref{le2.4}, we have
 \begin{equation}\label{2.4}\begin{array}{l}
 \sharp\Big\{\mbox{the\ connected\ components\ of\ the\ super-level\ set} ~ \{x\in\Omega : u(x)>t\} \Big\}= \sum\limits_{i = 1}^k {{m_i}}+1,
\end{array}\end{equation}
and
 \begin{equation}\label{2.5}\begin{array}{l}
 \sharp\Big\{\mbox{the\ connected\ components\ of\ the\ sub-level\ set} ~ \{x\in\Omega : u(x)<t\} \Big\}= \sum\limits_{i = 1}^k {{m_i}}+1.
\end{array}\end{equation}
On the other hand, using the results of step 4 and the strong maximum principle, therefore we obtain
\begin{equation}\label{2.6}\begin{array}{l}
 \sum\limits_{i = 1}^k {{m_i}}+1=N.
\end{array}\end{equation}
This completes the proof of Theorem \ref{th1.2}.
\end{proof}

Let $N=1$, we have:
\begin{Corollary}\label{co2.5}
 Suppose that $u$ is a non-constant solution of (\ref{1.2}) and that $\psi$ has exactly one maximal point on $\partial \Omega$, then $u$ has no interior critical points in $\Omega.$
\end{Corollary}

\section{The case of multiply connected domains: $\psi(x)\geq H$}
\subsection{Proof of Theorem \ref{th1.3}}
~~~~~In order to prove Theorem \ref{th1.3}, we need the following basic lemmas.
\begin{Lemma}\label{le3.1} 
Let $u$ be a non-constant solution of (\ref{1.5}). For any $t\in (H,\mathop{\max }\limits_{\gamma_E}\psi(x)),$ then any connected component of $\{x\in \Omega : u(x)>t\}$ has to meet the external boundary $\gamma_E.$
\end{Lemma}
\begin{proof}[Proof] Let $A$ be a connected component of $\{x\in \Omega : u(x)>t\}$. According to the assumption of $u|_{\gamma_I}=H$, we know that $A$ can not contain $\gamma_I$. Then the strong maximum principle and (\ref{1.5}) show that the connected component $A$ has to meet the external boundary $\gamma_E.$
\end{proof}

\begin{Lemma}\label{le3.2}
 Suppose that $x_0$ is an interior critical point of $u$ in $\Omega$ and that $m$ is the multiplicity of $x_0.$ Then $m+1$ distinct connected components of $\{x\in \Omega :u(x)>u(x_0)\}$ cluster around the point $x_0$.
\begin{proof}[Proof] According to the results of Hartman and Wintner \cite{Hartman}, in a neighborhood of $x_0$ the level line $\{x\in \Omega : u(x)=u(x_0)\}$ consists of $m+1$ simple arcs intersecting at $x_0$. By Lemma \ref{le3.1}, there exist $m+1$ distinct connected components of $\{x\in \Omega :u(x)>u(x_0)\}$ clustering around the point $x_0$. This completes the proof.
\end{proof}
\end{Lemma}

\begin{Lemma}\label{le3.3} 
If there exists $t\in (H, \mathop{\max }_{\gamma_E}\psi(x))$ such that a connected component $\omega$ of $\{x\in \Omega : u(x)<t\}$ is non-simply connected and the external boundary $\gamma$ of $\omega$ is a simply closed curve in $\Omega$, i.e., the external boundary  $\gamma$ of $\omega$ is a simply closed curve between $\gamma_I$ and $\gamma_E$. Then there does not exist any interior critical point in $\omega$.
\end{Lemma}
\begin{proof}[Proof]Suppose by contradiction that there exists an interior critical point $x_0$ in $\omega$ such that $H<u(x_0)<t$. Without loss of generality, we assume that the multiplicity of $x_0$ is one. According to the assumption of connected component $\omega$ and $u|_{\gamma_I}=H,$ then Lemma \ref{le3.2} implies the following distribution for the connected components of $\{x\in \omega :u(x)<u(x_0)\}.$
\begin{center}
  \includegraphics[width=5.5cm,height=3.0cm]{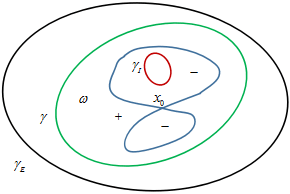}\\
  \scriptsize {\bf Figure 7.}~~The distribution for the connected components of $\{x\in \omega :u(x)<u(x_0)\}.$
\end{center}
By (\ref{1.5}) and  the strong maximum principle, this is impossible. Therefore there does not exist interior critical point in $\omega$.
\end{proof}

\begin{Lemma}\label{le3.4}
Suppose that $u$ is a non-constant solution to (\ref{1.5}). Then $u$ has finite interior critical points in $\Omega$.
\begin{proof}[Proof]We set up the usual contradiction argument. Suppose that $u$ has infinite interior critical points in $\Omega,$ denoting by $x_1,x_2,\cdots$. The results of Lemma \ref{le3.1}, Lemma \ref{le3.2} and Lemma \ref{le3.3} show that there exists infinite connected components of $\{x\in \Omega :u(x)>u(x_i)\}~(i=1,2,\cdots)$. The strong maximum principle implies that there exists at least a maximum point on $\gamma_E$ for any connected component of $\{x\in \Omega :u(x)>u(x_i)\}~(i=1,2,\cdots)$. Therefore there exists infinite maximal points on $\gamma_E,$ this contradicts with the assumption. This completes the proof.
\end{proof}
\end{Lemma}

\begin{Lemma}\label{le3.5}
Let $x_1,x_2,\cdots,x_k$ be the interior critical points of $u$ in $\Omega$. Suppose that $u(x_1)=u(x_2)=\cdots=u(x_k)\equiv t$ for some $t\in (H,\mathop{\max }_{\gamma_E}\psi(x))$ and that all the critical points $x_1,x_2,\cdots,x_k$ together with the corresponding level lines of $\{x\in \Omega : u(x)=t\}$ clustering round these points form $q$ connected sets, where $q\geq 1$ and $m_1,m_2,\cdots,m_k$ are the multiplicities of critical points $x_1,x_2,\cdots,x_k$ respectively.\\
Case 1: Suppose that there exists a non-simply connected component $\omega$ of $\{x\in \Omega : u(x)<t\}$ and the external boundary $\gamma$ of $\omega$ is a simply closed curve between $\gamma_I$ and $\gamma_E$ such that $u$ has at least one critical point on $\gamma$, then
 \begin{equation}\label{3.1}\begin{array}{l}
\sharp\Big\{\mbox{the\ simply\ connected\ components\ $\omega$ of\ the\ sub-level\ set} ~ \{x\in\Omega : u(x)<t\}\\~~~ \mbox{such\ that\ $\omega$ meet\ the\ external\ boundary\ $\gamma_E$} \Big\}= \sum\limits_{i = 1}^k {{m_i}}+q-1.
\end{array}\end{equation}
Case 2: Suppose that there exists a non-simply connected component $\omega$ of $\{x\in \Omega : u(x)<t\}$ such that $\omega$ meets $\gamma_E$. In addition, we set $M_1$ and $M_2$ as the number of the connected components of the super-level set $\{x\in\Omega : u(x)>t\}$ and the sub-level set $\{x\in\Omega : u(x)<t\}$, respectively. Then
\begin{equation}\label{3.2}\begin{array}{l}
M_1\geq \sum\limits_{i = 1}^k {{m_i}}+1,~M_2\geq \sum\limits_{i = 1}^k {{m_i}}+1, ~\mbox{and}~M_1+M_2=2\sum\limits_{i = 1}^k {{m_i}}+q+1.
\end{array}\end{equation}
\end{Lemma}

\begin{proof}[Proof] (i) Case 1: We divide the proof into two steps.

  Step 1: When $q=1,$ by induction. When $k=1,$ the result holds by Lemma \ref{le3.1} and Lemma \ref{le3.2}. Assume that $1\leq k\leq n$ the connected set, which consists of $k$ critical points and the connected components clustering round these points, contains exactly $\sum_{i = 1}^k {{m_i}}$ components $\omega$ of the sub-level set $\{x\in\Omega : u(x)<t\}$ such that $\omega$ meet the external boundary $\gamma_E$. Let $k=n+1.$ Let $A$ be the set which consists of the points $x_1,x_2,\cdots,x_{n+1}$ together with the respective components clustering round these points. We may assume that the points $x_1,x_2,\cdots,x_n$ together with the respective components clustering round these points form a connected set, denotes by $B.$ By Lemma \ref{le3.1} we know that $A$ cannot surround a component of $\{x\in\Omega : u(x)>t\}.$  Up to renumbering, therefore there is only one component of $\{x\in\Omega : u(x)<t\}$ whose boundary $\alpha$ contains both $x_n$ and $x_{n+1}$. By Lemma \ref{le3.3}, next we give the distribution for the level lines of $\{x\in\Omega : u(x)=t\}.$
\begin{center}
  \includegraphics[width=6cm,height=3.4cm]{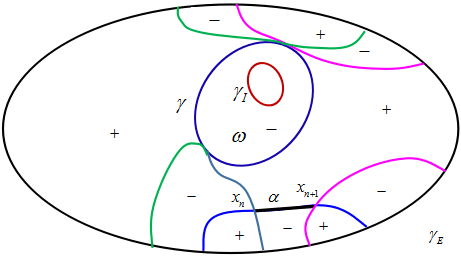}\\
  \scriptsize {\bf Figure 8.}~~The distribution for the level lines of $\{x\in\Omega : u(x)=t\}$.
\end{center}
 Since both $A$ and $B$ are connected. By using Lemma \ref{le3.2} and the inductive assumption to $B,$ then we know that $A$ contains exactly
 $$ \sum\limits_{i = 1}^n {{m_i}}+(m_{n+1}+1)-1=\sum\limits_{i = 1}^{n+1} {{m_i}}$$
connected components $\omega$ of the sub-level set $\{x\in\Omega : u(x)<t\}$ such that $\omega$ meet the external boundary $\gamma_E$. This completes the proof of step 1.

Step 2: When $q\geq 2.$ Since the number of connected sets of the level lines $\{x\in\Omega : u(x)=t\}$ together with $x_1,\cdots,x_k$ increases one leading the number of connected components of $\{x\in\Omega : u(x)<t\}$ increases one. If all the critical points $x_1,x_2,\cdots,x_k$ together with the level lines $\{x\in\Omega : u(x)=t\}$ clustering round these points form $q$ connected sets. By the results of step 1, then we have
\begin{equation*}\begin{array}{l}
 \sharp\Big\{\mbox{the\ simply\ connected\ components\ $\omega$ of\ the\ sub-level\ set} ~ \{x\in\Omega : u(x)<t\}\\~~~ \mbox{such\ that\ $\omega$ meet\ the\ external\ boundary\ $\gamma_E$} \Big\}= \sum\limits_{i = 1}^k {{m_i}}+(q-1).
\end{array}\end{equation*}
This completes the proof of case 1.

(ii) Case 2: We divide the proof of case 2 into two steps.

  Step 1: When $q=1$, by induction. When $k=1,$ the result holds by Lemma \ref{le3.1} and Lemma \ref{le3.2}. Assume that $1\leq k\leq n$ the connected set, which consists of $k$ critical points and the connected components clustering round these points, contains exactly $\sum\limits_{i = 1}^k {{m_i}}+1$ components of the super-level set $\{x\in\Omega : u(x)>t\}$. Let $k=n+1.$ Let $A$ be the set which consists of the points $x_1,x_2,\cdots,x_{n+1}$ together with the respective components clustering round these points. We may assume that the points $x_1,x_2,\cdots,x_n$ together with the respective components clustering round these points form a connected set, denotes by $B.$ By Lemma \ref{le3.1} we know that $A$ cannot surround a component of $\{x\in\Omega : u(x)<t\}.$  Up to renumbering, therefore there is only one component of $\{x\in\Omega : u(x)>t\}$ whose boundary $\gamma$ contains both $x_n$ and $x_{n+1}$. Next we give the distribution for the level lines of $\{x\in\Omega : u(x)=t\}.$
\begin{center}
  \includegraphics[width=6cm,height=3.6cm]{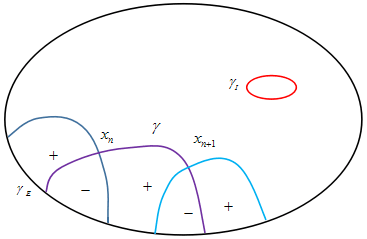}\\
  \scriptsize {\bf Figure 9.}~~The distribution for the level lines of $\{x\in\Omega : u(x)=t\}$.
\end{center}
 Since both $A$ and $B$ are connected. By using Lemma \ref{le3.2} and the inductive assumption to $B,$ then we know that $A$ contains exactly
 $$ (\sum\limits_{i = 1}^n {{m_i}}+1)+(m_{n+1}+1)-1=\sum\limits_{i = 1}^{n+1} {{m_i}}+1$$
connected components of the super-level set $\{x\in\Omega : u(x)>t\}$.

Step 2: The proof is similar to the case 2 of Lemma \ref{le2.4}. When $q\geq 2.$ Since the theorem of Hartman and Wintner \cite{Hartman} shows that the interior critical points of solution $u$ are isolated, so the number of connected sets of the level lines $\{x\in \Omega : u(x)=t\}$ together with $x_1,\cdots,x_k$ increases one leading the number of connected components of $\{x\in\Omega : u(x)>t\}$ or $\{x\in\Omega : u(x)<t\}$ increases one. If all the critical points $x_1,x_2,\cdots,x_k$ together with the level lines of $\{x\in \Omega : u(x)=t\}$ clustering round these points form $q$ connected sets. By the results of step 1, then we have
\begin{equation*}\begin{array}{l}
M_1\geq \sum\limits_{i = 1}^k {{m_i}}+1,~~~M_2\geq \sum\limits_{i = 1}^k {{m_i}}+1,
\end{array}\end{equation*}
and
\begin{equation*}\begin{array}{l}
M_1+M_2=2(\sum\limits_{i = 1}^k {{m_i}}+1)+(q-1)=2\sum\limits_{i = 1}^k {{m_i}}+q+1.
\end{array}\end{equation*}
This completes the proof of case 2.
\end{proof}

 \begin{Remark}\label{re3.6}  Note that if all critical values are equal (i.e., $u(x_1)=\cdots=u(x_k)\equiv t$) and there exists a non-simply connected component $\omega$ of $\{x\in \Omega : u(x)<t\}$ for critical value $t$, where the external boundary $\gamma$ of $\omega$ is a simply closed curve in $\Omega$ as in Lemma \ref{le3.3}, then $u$ has at least one critical point on $\gamma$. In fact, suppose by contradiction that $u$ has no critical point on $\gamma$. Without loss of generality, we may assume that $\psi(x)$ has only two local maximal points $q_1, q_2$ on $\gamma_E$ and one critical point $x_1$ in $\Omega\setminus \overline{\omega}$ such that $u(x_1)=t$ and the multiplicity of $x_1$ is one, we denote the non-simply connected component of $\{x\in\Omega: u(x)>t\}$ by $A$. The distribution for the level lines of $\{x\in\Omega: u(x)=t\}$ as follows:
 \begin{center}
  \includegraphics[width=5.5cm,height=3.5cm]{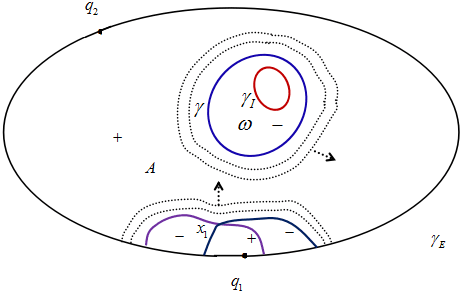}\\
  \scriptsize {\bf Figure 10.}~~The distribution for the level lines of $\{x\in\Omega : u(x)=t\}$.
\end{center}
 By using the method of step 3 of the proof of Theorem 1.2, this would imply that either: $u(x)=u(q_2)$ in interior points of $A$, or: there exists two level lines intersect in $A$, i.e., there exists critical points in $A$. This is a contradiction.
\end{Remark}

We are now prepare to prove Theorem \ref{th1.3}.
\begin{proof}[Proof of Theorem \ref{th1.3}]  (i) Case 1: If $u(x_1)=u(x_2)=\cdots=u(x_k)\equiv t$ for some $t\in (H,\mathop{\max }_{\gamma_E}\psi(x)).$ By the results of case 1 of Lemma \ref{le3.5}, we know that
\begin{equation*}\begin{array}{l}
 \sharp\Big\{\mbox{the\  simply\ connected\ components\ of\ the\ sub-level\ set} ~ \{x\in\Omega : u(x)<t\} \Big\}\geq \sum\limits_{i = 1}^k {{m_i}}.
\end{array}\end{equation*}
Therefore, in this case the sub-level set always has at least $\sum\limits_{i = 1}^k {{m_i}}$ connected components $\omega$ such that $\omega$ meet $\gamma_E$. Using the strong maximum principle, we have that $u$ exists at least $\sum\limits_{i = 1}^k {{m_i}}$ local minimal points on $\gamma_E.$ Hence, we have
\[ \sum\limits_{i = 1}^k {{m_i}}\leq N.\]

(ii) Case 2: The values at critical points $x_1,\cdots,x_k$ are not totally equal. Next we need divide the proof of case 2 into two situations.

(1) Situation 1: If there exists a non-simply connected component $\omega$ of $\{x\in \Omega : u(x)<t\}$ for some $t\in (H,\mathop{\max }_{\gamma_E}\psi(x))$ such that $\omega$ meets $\gamma_E$. Without loss of generality, we may suppose that
\begin{equation}\label{3.3}
\begin{split}
u(x_1)&=\cdots=u(x_{j_1})<u(x_{j_1+1})=\cdots=u(x_{j_2})<\cdots<\cdots\\
&<u(x_{j_{n-1}+1})=\cdots=u(x_{j_{n}}),
\end{split}
\end{equation}
where $x_1,\cdots,x_{j_1},\cdots, x_{j_2},\cdots, x_{j_{n}}$ are different critical points in $\Omega$, $j_{n}=k$ and $n\geq 2.$ Now we put
\begin{equation*}\begin{array}{l}
E_{j}:=\Big\{\omega : \mbox{open\ set}~ \omega ~ \mbox{is\ a\ connected\ component\ of} ~ \{x\in \Omega: u(x)>u(x_j)\}\Big\}~(j=j_1,j_2,\cdots,j_n),
\end{array}\end{equation*}
and
\begin{equation*}\begin{array}{l}
G_{j_n}:=\Big\{\omega : \mbox{open\ set}~ \omega~ \mbox{is\ a\ connected\ component\ of} ~ \{x\in \Omega: u(x)>u(x_{j})\}(j=j_1,j_2,\cdots,j_n)\\~~~~~~~~~~~~ \mbox{such\ that\ there\ does\ not\ exist\ interior\ critical\ point\ in} ~\omega\Big\}.
\end{array}\end{equation*}
According to the definition, we know that $G_{j_n}$ consists of disjoint components. Denotes
\[|G_{j_n}|:=\sharp\{\omega : \omega~ \mbox{is\ a\ connected\ component\ of} ~G_{j_n}\}.\]

To illustrate $G_{j_n}$, let us consider an illustration for $G_{j_2}$. Assume that $u$ has only three critical points $x_{j_1},x_{j_1+1},x_{j_2}$ with respective multiplicity $m_{j_1}=1,m_{j_1+1}=1,m_{j_2}=2$ in $\Omega$ and $u(x_{j_1})<u(x_{j_1+1})=u(x_{j_2})$. The distribution of elements of $G_{j_2}$ as follows:
\begin{center}
  \includegraphics[width=6.6cm,height=3.8cm]{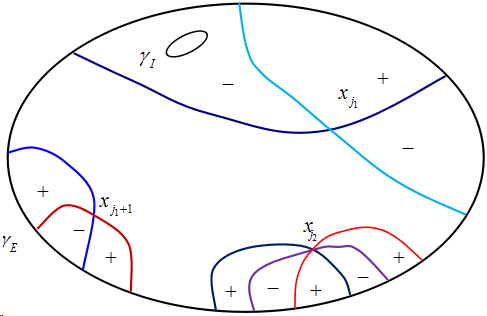}\\
  \scriptsize {\bf Figure 11.}~~The distribution of elements of $G_{j_2}$.
\end{center}
By the definition of $G_{j_n},$ we know $|G_{j_2}|=6.$

Now let us show that $\big|G_{j_s}\big|\geq \sum\limits_{i = 1}^{{j_s}} m_i+1$ by induction on the number $s.$ When $s=1,$ the result holds by case 1. Assume that $\big|G_{j_s}\big|\geq \sum\limits_{i = 1}^{{j_s}} m_i+1$ for $1\leq s\leq n-1.$  Let $s=n.$ Then, by (\ref{3.4}) and the definition of $E_{j}$, we have
\[\{x_{j_{n-1}+1},\cdots,x_{j_{n}}\}\subset \bigcup \limits_{\omega  \in {E_{{j_{n-1}}}}} \omega,\]
where $\omega$ is a connected component of $\{x\in \Omega: u(x)>u(x_{j_{n-1}})\}.$

Let us assume that $\{x_{j_{n-1}+1},\cdots,x_{j_{n}}\}$ are contained in exactly $\overline{q}$ components $\omega_1,\cdots,\omega_{\overline{q}}.$ Then $x_{j_{n-1}+1},\cdots,x_{j_{n}}$ together with the corresponding level lines of $\{x\in \Omega: u(x)=u(x_{j_{n}})\}$ clustering round these points at least form $\overline{q}$ connected sets. By the case 3 of Lemma \ref{le3.5}, we have
\begin{equation*}\begin{array}{l}
\widetilde{M}:=\sharp\Big\{\mbox{the\ connected\ components\ of}~ \{x\in \Omega: u(x)>u(x_{j_{n}})\}~\mbox{in\ all}~\omega_j~(j=1,\cdots,\overline{q})\Big\}\\ ~~~~\geq \sum\limits_{i=j_{n-1}+ 1}^{j_{n}} {{m_i}}+\overline{q}.
\end{array}\end{equation*}
By using the definition of $\big|G_{j_{n}}\big|$ and the inductive assumption to $1\leq s\leq n-1,$ since $\{x_{j_{n-1}+1},\cdots,x_{j_{n}}\}$ are contained in exactly $\overline{q}$ components $\omega_1,\cdots,\omega_{\overline{q}},$ so when we calculate the number of $|G_{j_n}|$, the number of $|G_{j_{n-1}}|$ will be reduced by $\overline{q}$. Then we have
\begin{equation*}\begin{array}{l}
\big|G_{j_{n}}\big|=\big|G_{j_{n-1}}\big|+\widetilde{M}-\overline{q}\geq \big|G_{j_{n-1}}\big|+(\sum\limits_{i=j_{n-1}+ 1}^{j_{n}} {{m_i}}+\overline{q})-\overline{q}\geq\sum\limits_{i=1}^{j_{n}} {{m_i}}+1.
\end{array}\end{equation*}
By the strong maximum principle and Lemma \ref{le3.1}, we have that $u$ has at least $\sum\limits_{i = 1}^k {{m_i}}+1$ local maximal points on $\gamma_E.$ Therefore, we obtain
\[ \sum\limits_{i = 1}^k {{m_i}}+1\leq N.\]

(2) Situation 2: Suppose that there exists a non-simply connected component $\omega$ of $\{x\in \Omega : u(x)<t\}$ for some $t\in (H,\mathop{\max }_{\gamma_E}\psi(x))$ and the external boundary $\gamma$ of $\omega$ is a simply closed curve between $\gamma_I$ and $\gamma_E$ such that $u$ has at least one critical point on $\gamma$. The idea of proof is essentially same as the case 2 of the proof of Theorem \ref{th1.1} and the situation 1 of the proof of Theorem \ref{th1.3}. Here we omit the proof. This completes the proof of case 2.
\end{proof}

\subsection{Proof of Theorem \ref{th1.4}}
~~~~~In this subsection, we investigate the geometric structure of interior critical points of a solution in a planar, bounded, smooth annular domain $\Omega$ with the interior boundary $\gamma_I$ and the external boundary $\gamma_E$ for the case of $\psi$ having only $N$ global maximal points and $N$ global minimal points on $\gamma_E,$ where $N\geq 2.$ Next we show (\ref{1.8}) or (\ref{1.7}) by proving the three ``just right''s.

\begin{proof}[Proof of Theorem \ref{th1.4}] We divide the proof into two cases.

(1) Case 1: If there exists a non-simply connected component $\omega$ of $\{x\in \Omega : u(x)<t\}$ for some $t\in (H,\mathop{\max }_{\gamma_E}\psi(x))$ such that $\omega$ meets $\gamma_E$. Next we need divide the proof of case 1 into five steps.

Step 1, we should show that $u$ has at least one interior critical point in $\Omega.$ Suppose by contradiction that $|\nabla u|>0$ in $\Omega.$ Next, the idea of proof is essentially same as the step 1 in the proof of Theorem \ref{th1.2}. So we omit the proof.

Step 2, the first ``just right'': According to Lemma \ref{le3.4}, we assume that the interior critical points of $u$ are $x_1,x_2,\cdots,x_k$. We show that $u(x_1)=u(x_2)=\cdots=u(x_k)=t~\mbox{for\ some} ~ t\in \mathbb{R},$ i.e., we exclude the case of case 2 in Theorem \ref{th1.3}. According to the assumption of Theorem \ref{th1.4}, let $q_1,\cdots,q_N$ and $p_1,\cdots,p_N$ be respectively the global maximal points and minimal points on $\gamma_E.$ We set up the usual contradiction argument. We assume that the values at critical points $x_1,\cdots,x_k$ are not totally equal. Without loss of generality, we suppose that $u(x_1)<u(x_2)$ and that $m_1,m_2$ are the respective multiplicities of $x_1,x_2.$ Then, by Lemma \ref{le3.1}, we know that any connected component of $\{x\in \Omega:u(x)>u(x_1)\}$ has to meet the boundary $\gamma_E$ and $u(p_1)<u(x_1)$, where $p_1$ is the minimal point of some one connected component of $\{x\in \Omega:u(x)<u(x_1)\}$ on $\gamma_E$. At the same time, we know that any connected component $C$ of $\{x\in \Omega:u(x_1)<u(x)<u(x_2)\}$ has to meet the boundary $\gamma_E$ and $u(x_1)<u(p_2)<u(x_2)$, where $p_2$ is the minimal point of connected component $C$ on $\gamma_E$ (see Figure 12). Then $u(p_1)\neq u(p_2)$, which contradicts with the assumption of Theorem \ref{th1.4}. This completes the proof of step 2.
\begin{center}
  \includegraphics[width=6.0cm,height=3.7cm]{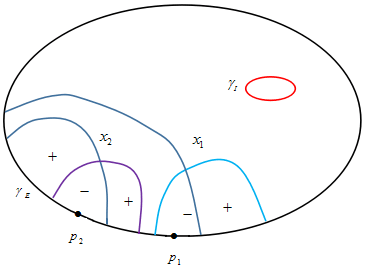}\\
  \scriptsize {\bf Figure 12.} ~~The distribution of the connected components.
\end{center}

Step 3, the second ``just right'': we show that $x_1,x_2,\cdots,x_k$ together with the corresponding level lines of $\{x\in \Omega : u(x)=t\}$ clustering round these points exactly form one connected set. Without loss of generality, we suppose by contradiction that $x_1,x_2,\cdots,x_k$ together with the level lines of $\{x\in \Omega : u(x)=t\}$ clustering round these points  form two connected sets.
Therefore, there exists a connected components of $\{x\in \Omega: u(x)<t\}$, which meets two parts $\gamma_1,\gamma_2$ of $\gamma_E$, denoting by $A$ (see Figure 13).
\begin{center}
  \includegraphics[width=6.0cm,height=3.7cm]{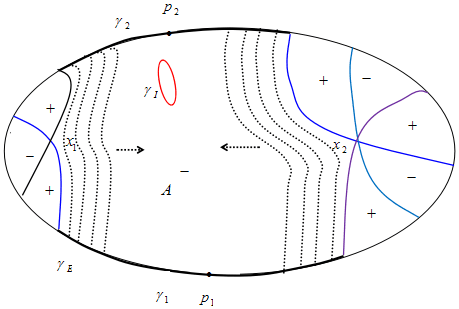}\\
  \scriptsize {\bf Figure 13.}~~The distribution of some level lines $\{x\in \Omega:u(x)=\widetilde{t},H<\widetilde{t}<t \}.$
\end{center}
Note that $u$ is monotonically decreasing on the connected components of boundary $\gamma_E$ from one maximal point to the near minimal point. Therefore, $\{x\in A: u(x)=t-\epsilon\}$ exactly exists two level lines in $A$ for any $\epsilon$ such that $0<\epsilon<t-H$. This is impossible, because this would imply that either: $u(x)=\mathop {\min}\limits_{\gamma_E}\psi(x)=H~\mbox{in\ interior\ points\ of}~A$, or: there exists two level lines intersect in $A$, i.e., there exists critical points in $A.$ This completes the proof of step 3.

Step 4, the third ``just right'': we show that every simply connected component of $\{x\in \Omega: u(x)>t\}$ has exactly one global maximal point on boundary $\gamma_E.$ In fact, we assume that some simply connected component $B$ of $\{x\in \Omega: u(x)>t\}$ exists two global maximal points on boundary $\gamma_E.$ According to $\psi$ has only $N$ global maximal points and $N$ global minimal points on $\gamma_E$, then there must exist a minimal point $\widetilde{p}$ between the two maximal points on $\gamma_E$ such that $u(\widetilde{p})=H$. Since $u(x)>t>H$ in $B$, by the continuity of solution $u$, this contradicts with the definition of connected component $B$. This completes the proof of step 4.

Step 5, By the results of step 3 and the results of case 2 in Lemma \ref{le3.5}, we have
 \begin{equation}\label{3.4}\begin{array}{l}
 \sharp\Big\{\mbox{the\ simply\ connected\ components\ of\ the\ super-level\ set} ~ \{x\in\Omega : u(x)>t\} \Big\}= \sum\limits_{i = 1}^k {{m_i}}+1.
\end{array}\end{equation}
On the other hand, using the results of step 4 and the strong maximum principle, therefore we obtain
\begin{equation}\label{3.5}\begin{array}{l}
 \sum\limits_{i = 1}^k {{m_i}}+1=N.
\end{array}\end{equation}

(2) Case 2: Suppose that there exists a non-simply connected component $\omega$ of $\{x\in \Omega : u(x)<t\}$ for some $t\in (H,\mathop{\max }_{\gamma_E}\psi(x))$ and the external boundary $\gamma$ of $\omega$ is a simply closed curve between $\gamma_I$ and $\gamma_E$ such that $u$ has at least one critical point on $\gamma$. The idea of proof is essentially same as the proof of case 1. Next we need divide the proof of case 2 into four steps.

Step 1, the first ``just right'': According to Lemma \ref{le3.4}, we assume that the interior critical points of $u$ are $x_1,x_2,\cdots,x_k$. We show that $u(x_1)=u(x_2)=\cdots=u(x_k)=t~\mbox{for\ some} ~ t\in \mathbb{R},$ i.e., we exclude the case of case 2 in Theorem \ref{th1.3}. The proof is same as the step 2 of the proof of case 1.

Step 2, the second ``just right'': we show that $x_1,x_2,\cdots,x_k$ together with the corresponding level lines of $\{x\in \Omega: u(x)=t\}$ clustering round these points exactly form one connected set. The proof is same as the step 3 of the proof of case 1.

Step 3, the third ``just right'': we show that every simply connected component $\omega$ of $\{x\in \Omega: u(x)<t\}$ has exactly one global minimal point on $\gamma_E,$ where $\omega$ meets the external boundary $\gamma_E.$  In fact, we assume that some simply connected component $\omega$ of $\{x\in \Omega: u(x)<t\}$ exists two global minimal points on boundary $\gamma_E.$ According to $\psi$ has only $N$ global maximal points and $N$ global minimal points on $\gamma_E$, then there must exist a maximal point $\overline{p}$ between the two minimal points on $\gamma_E$ such that $u(\overline{p})=\max_{\gamma_E} \psi(x)$. Since $u(x)<t<\max_{\gamma_E} \psi(x)$ in $\omega$, by the continuity of solution $u$, this contradicts with the definition of connected component $\omega$. This completes the proof of step 3.

Step 4, By the results of step 2 and the results of step 1 of case 1 in Lemma \ref{le3.5}, we have
 \begin{equation}\label{3.6}\begin{array}{l}
 \sharp\Big\{\mbox{the\ simply\ connected\ components\ $\omega$ of\ the\ sub-level\ set} ~ \{x\in\Omega : u(x)<t\}\\~~~ \mbox{such\ that\ $\omega$ meet\ the\ external\ boundary\ $\gamma_E$} \Big\}=\sum\limits_{i = 1}^k {{m_i}}.
\end{array}\end{equation}
On the other hand, using the results of step 3 and the strong maximum principle, therefore we obtain
\begin{equation}\label{3.7}\begin{array}{l}
 \sum\limits_{i = 1}^k {{m_i}}=N.
\end{array}\end{equation}
This completes the proof of Theorem \ref{th1.4}.
\end{proof}

Let $N=1$, we have:
\begin{Corollary}\label{co3.6}
 Suppose that $u$ is a non-constant solution of (\ref{1.5}) and that $\psi$ has exactly one maximal point on $\gamma_E$. Then $u$ has at most one interior critical point $p$ in $\Omega.$ If $u$ has one interior critical point $p$, then the multiplicity of the interior critical point $p$ is one.
\end{Corollary}

According to Theorem \ref{th1.3} and Theorem \ref{th1.4}, we can easily have the following results.
\begin{Corollary}\label{co3.7}
Let $\Omega$ be a bounded smooth annular domain with the interior boundary $\gamma_I$ and the external boundary $\gamma_E$ in $\mathbb{R}^{2}$. Suppose that $\psi(x)\in C^1(\overline{\Omega}),$ $H$ is a given constant, $\psi(x)\leq H.$ Let $u$ be a non-constant solution of (\ref{1.5}). Then we have:

(i) If $\psi$ has $N$ local minimal points on $\gamma_E,$ then $u$ has finite interior critical points and inequality (\ref{1.6}) holds;

(ii) If $\psi$ has only $N$ global minimal points and $N$ global maximal points on $\gamma_E,$ i.e., all the maximal and minimal points of $\psi$ are global, then $u$ has finite interior critical points and equality (\ref{1.7}) or  (\ref{1.8}) holds;

(iii) If $\psi$ has exactly one minimal point on $\gamma_E.$ Then $u$ has at most one interior critical point $p$ in $\Omega.$ If $u$ has one interior critical point $p$, then the multiplicity of the interior critical point $p$ is one.
\end{Corollary}

\section{The case of multiply connected domains: $\mathop {\min}\limits_{\gamma_E}\psi(x)<H<\mathop {\max}\limits_{\gamma_E}\psi(x)$}
\subsection{Proof of Theorem \ref{th1.5}}
~~~~~~~In this subsection, we put $z:=\mathop{\min }\limits_{\gamma_E}\psi(x),~Z:=\mathop{\max }\limits_{\gamma_E}\psi(x).$ In order to prove Theorem \ref{th1.5}, we need the following basic lemmas.
\begin{Lemma}\label{le4.1} 
Let $u$ be a non-constant solution of (\ref{1.5}), constant $H$ satisfies $z<H<Z.$\\
(1) For any  $t\in (H,Z),$ then any connected component $\omega$ of $\{x\in \Omega : u(x)>t\}$ has to meet the external boundary $\gamma_E.$\\
(2) For any  $t\in (z,H],$ then the connected component $\omega$ of $\{x\in \Omega : u(x)>t\}$ is simply connected or non-simply connected. If $\omega$ is simply connected, which has to meet the external boundary $\gamma_E.$
\end{Lemma}
\begin{proof}[Proof](1) Since $t\in (H,Z),$ the proof is same as the proof of Lemma \ref{le3.1}. The results of (2) naturally holds. In fact, by the strong maximum principle and $u|_{\gamma_I}=H$, any connected component $\omega$ of $\{x\in \Omega : u(x)>H\}$ or $\{x\in \Omega : u(x)<H\}$ has to meet the external boundary $\gamma_E.$
\end{proof}

\begin{Lemma}\label{le4.2}
 Suppose that $x_0$ is an interior critical point of $u$ in $\Omega$ and that $m$ is the multiplicity of $x_0.$ Then $m+1$ distinct connected components of $\{x\in \Omega :u(x)>u(x_0)\}$ cluster around the point $x_0$.
 \end{Lemma}
\begin{proof}[Proof] According to the results of Hartman and Wintner \cite{Hartman}, in a neighborhood of $x_0$ the level line $\{x\in \Omega : u(x)=u(x_0)\}$ consists of $m+1$ simple arcs intersecting at $x_0$. By the results of Lemma \ref{le4.1}, there exist $m+1$ distinct connected components of $\{x\in \Omega :u(x)>u(x_0)\}$ clustering around the point $x_0$. This completes the proof.
\end{proof}

\begin{Lemma}\label{le4.3} 
 If there exists $t\in (z,H)$ such that a connected component $\omega$ of $\{x\in \Omega : u(x)>t\}$ is non-simply connected and the external boundary of $\omega$ is a simply closed curve between $\gamma_I$ and $\gamma_E$. Then there does not exist any interior critical point in $\omega$.
\end{Lemma}
\begin{proof}[Proof] The proof is same as the proof of Lemma \ref{le3.3}, so we omit the proof.
\end{proof}

\begin{Lemma}\label{le4.4}
Suppose that $u$ is a non-constant solution to (\ref{1.5}). Then $u$ has finite interior critical points in $\Omega$.
\end{Lemma}
\begin{proof}[Proof]We set up the usual contradiction argument. Suppose that $u$ has infinite interior critical points in $\Omega,$ denoting by $x_1,x_2,\cdots$. The results of Lemma \ref{le4.1} and Lemma \ref{le4.2} show that there exists infinite simply connected components of $\{x\in \Omega :u(x)>u(x_i)\}~(i=1,2,\cdots)$, which meet the external boundary $\gamma_E$. The strong maximum principle implies that there exists at least a maximum point on $\gamma_E$ for any simply connected component $\omega$ of $\{x\in \Omega :u(x)>u(x_i)\}$ such that $\omega$ meet the external boundary $\gamma_E$. Therefore there exists infinite maximal points on $\gamma_E,$ this contradicts with the assumption. This completes the proof.
\end{proof}

\begin{Lemma}\label{le4.5}
Let $x_1,x_2,\cdots,x_k$ be the interior critical points of $u$ in $\Omega$. Suppose that $u(x_1)=u(x_2)=\cdots=u(x_k)\equiv t$ and that all the critical points $x_1,x_2,\cdots,x_k$ together with the corresponding level lines of $\{x\in \overline{\Omega} : u(x)=t\}$ clustering round these points form $q$ connected sets, where $q\geq 1$ and $m_1,m_2,\cdots,m_k$ are the multiplicities of critical points $x_1,x_2,\cdots,x_k$ respectively.\\
Case 1: Suppose that there exists a non-simply connected component $\omega$ of $\{x\in \Omega : u(x)>t\}$ for some $t\in (z,H)$ and the external boundary $\gamma$ of $\omega$ is a simply closed curve between $\gamma_I$ and $\gamma_E$ such that $u$ has at least one critical point on $\gamma$, then
 \begin{equation}\label{4.1}\begin{array}{l}
\sharp\Big\{\mbox{the\ simply\ connected\ components\ $\omega$ of\ the\ super-level\ set} ~ \{x\in\Omega : u(x)>t\}\\~~~ \mbox{such\ that\ $\omega$ meet\ the\ external\ boundary\ $\gamma_E$} \Big\}= \sum\limits_{i = 1}^k {{m_i}}+q-1.
\end{array}\end{equation}
Case 2: Suppose  that $t=H$ or that there exists a non-simply connected component $\omega$ of $\{x\in \Omega : u(x)>t\}$ for some $t\in (z,H)$ such that $\omega$ meets $\gamma_E$. In addition, we set $M_1$ and $M_2$ as the number of the connected components of the super-level set $\{x\in\Omega : u(x)>t\}$ and the sub-level set $\{x\in\Omega : u(x)<t\}$, respectively. Then
\begin{equation}\label{4.2}\begin{array}{l}
M_1\geq \sum\limits_{i = 1}^k {{m_i}}+1,~M_2\geq \sum\limits_{i = 1}^k {{m_i}}+1, ~\mbox{and}~M_1+M_2=2\sum\limits_{i = 1}^k {{m_i}}+q+1.
\end{array}\end{equation}
Case 3: Suppose that $t\in (H,Z),$ the results see Lemma \ref{le3.5}.
\end{Lemma}

\begin{proof}[Proof](i) Case 1: The proof is same as the proof of case 1 of Lemma \ref{le3.5}.

(ii) Case 2: Lemma \ref{le2.4} and the case 2 of Lemma \ref{le3.5} implies case 2.
\end{proof}

We are now ready to prove Theorem \ref{th1.5}.
\begin{proof}[Proof of Theorem \ref{th1.5}]  (i) Case 1: If $u(x_1)=u(x_2)=\cdots=u(x_k)\equiv t$. By the results of Lemma \ref{le4.5}, we know that
\begin{equation*}\begin{array}{l}
 \sharp\Big\{\mbox{the\ simply\ connected\ components\ $\omega$ of\ the\ super-level\ set} ~ \{x\in\Omega : u(x)>t\}\\~~~ \mbox{for\ $t\in(z,H]$\ such\ that\ $\omega$ meet\ the\ external\ boundary\ $\gamma_E$} \Big\}\geq \sum\limits_{i = 1}^k {{m_i}},
\end{array}\end{equation*}
or
\begin{equation*}\begin{array}{l}
 \sharp\Big\{\mbox{the\ simply\ connected\ components\ $\omega$ of\ the\ sub-level\ set} ~ \{x\in\Omega : u(x)<t\}\\~~~ \mbox{for\ $t\in(H,Z)$\ such\ that\ $\omega$ meet\ the\ external\ boundary\ $\gamma_E$} \Big\}\geq \sum\limits_{i = 1}^k {{m_i}},
\end{array}\end{equation*}
By the strong maximum principle, we have that $u$ exists at least $\sum\limits_{i = 1}^k {{m_i}}$ local maximal points or minimal points on $\gamma_E.$ Hence, we have
\[ \sum\limits_{i = 1}^k {{m_i}}\leq N.\]
(ii) Case 2: The values at critical points $x_1,\cdots,x_k$ are not totally equal. The idea of proof is essentially same as the case 2 of the proof of Theorem \ref{th1.1} and Theorem \ref{th1.3}. Here we omit the proof.
\end{proof}

\subsection{Proof of Theorem \ref{th1.6}}
~~~~~In this subsection, we investigate the geometric structure of interior critical points of a solution in a planar, bounded, smooth annular domain $\Omega$ with the interior boundary $\gamma_I$ and the external boundary $\gamma_E$ for the case of $\mathop {\min}\limits_{\gamma_E}\psi(x)<H<\mathop {\max}\limits_{\partial \Omega}\psi(x)$ and $\psi$ having only $N$ global maximal points and $N$ global minimal points on $\gamma_E$, where $N\geq2$. Next we show (\ref{1.11}) or (\ref{1.10}) by proving the three ``just right''s.
\begin{proof}[Proof of Theorem \ref{th1.6}] We divide the proof into three cases.

(1) Case 1: Suppose that there exists a non-simply connected component $\omega$ of $\{x\in \Omega : u(x)>t\}$ for some $t\in (z,H)$ such that $\omega$ meets $\gamma_E$ or that there exists critical value $H$. The idea of proof is essentially same as the proof of Theorem \ref{th1.2} and the case 1 of Theorem \ref{th1.4}. Then we have
\begin{equation}\label{4.3}\begin{array}{l}
\begin{array}{l}
\sum\limits_{i = 1}^k {{m_i}}+1= N,
\end{array}
\end{array}\end{equation}
where $m_1,m_2,\cdots,m_k$ are the multiplicities of critical points $x_1,x_2,\cdots,x_k$ respectively.

(2) Case 2:  Suppose that there exists a non-simply connected component $\omega$ of $\{x\in \Omega : u(x)>t\}$ for some $t\in (z,H)$ and the external boundary $\gamma$ of $\omega$ is a simply closed curve between $\gamma_I$ and $\gamma_E$ such that $u$ has at least one critical point on $\gamma$. We deduce (\ref{1.10}) by proving the three ``just right''s. We should divide the proof of case 2 into four steps.

Step 1, the first ``just right'': According to Lemma \ref{le4.4}, we assume that the interior critical points of $u$ are $x_1,x_2,\cdots,x_k$. We show that $u(x_1)=u(x_2)=\cdots=u(x_k)\equiv t,$ i.e., we exclude the case of case 2 in Theorem \ref{th1.5}. The proof is same as the step 2 of the proof of case 1 of Theorem \ref{th1.4}.

Step 2, the second ``just right'': we show that $x_1,x_2,\cdots,x_k$ together with the corresponding level lines of $\{x\in \Omega: u(x)=t\}$ clustering round these points exactly form one connected set. The proof is same as the step 3 of the proof of case 1 of Theorem \ref{th1.4}.

Step 3, the third ``just right'': we show that every simply connected component $\omega$ of $\{x\in \Omega: u(x)>t\}$ has exactly one global maximal point on boundary $\gamma_E,$ where $\omega$ meets the external boundary  $\gamma_E.$  In fact, we assume that some simply connected component $B$ of $\{x\in \Omega: u(x)>t\}$ exists two global maximal points on boundary $\gamma_E.$ According to $\psi$ has only $N$ global maximal points and $N$ global minimal points on $\gamma_E$, then there must exist a minimal point $\overline{p}$ between the two maximal points on $\gamma_E$ such that $u(\overline{p})=z$. Since $u(x)>t>z$ in $B$, by the continuity of solution $u$, this contradicts with the definition of connected component $B$. This completes the proof of step 3.

Step 4, By the results of step 2 and case 1 of Lemma \ref{le4.5}, we have
 \begin{equation}\label{4.4}\begin{array}{l}
 \sharp\Big\{\mbox{the\ simply\ connected\ components\ $\omega$ of\ the\ super-level\ set} ~ \{x\in\Omega : u(x)>t\}\\~~~ \mbox{such\ that\ $\omega$ meet\ the\ external\ boundary\ $\gamma_E$} \Big\}=\sum\limits_{i = 1}^k {{m_i}}.
\end{array}\end{equation}
On the other hand, using the results of step 3 and the strong maximum principle, therefore we obtain
\begin{equation}\label{4.5}\begin{array}{l}
 \sum\limits_{i = 1}^k {{m_i}}=N.
\end{array}\end{equation}
Case 3: For $t\in (H,Z)$, the results see Theorem \ref{th1.4}. This completes the proof of Theorem \ref{th1.6}.
\end{proof}

Let $N=1$, we have:
\begin{Corollary}\label{co4.6}
 Suppose that $\mathop {\min}\limits_{\gamma_E}\psi(x)<H<\mathop {\max}\limits_{\partial \Omega}\psi(x)$ and that $\psi$ has exactly one maximal point on $\gamma_E$, Let $u$ be a non-constant solution of (\ref{1.5}). Then $u$ has at most one interior critical point $p$ in $\Omega.$ If $u$ has one interior critical point $p$, then the multiplicity of the interior critical point $p$ is one.
\end{Corollary}

\noindent  \textbf{Acknowledgement.} We are very grateful to the anonymous referees for the very careful reading and many very valuable suggestions which have helped to improve the presentation of this paper and the first author is very grateful to his advisor Professor Xiaoping Yang for his expert guidances and useful conversations.

\end{document}